\documentclass[10pt,twoside]{siamltex}
\usepackage{amsfonts,epsfig}
\usepackage{amssymb}

\usepackage{xy}
\input{xy}
\xyoption{all} \xyoption{rotate}

\newtheorem{remark}[theorem]{Remark}

\newtheorem{example}[theorem]{Example}

\begin{document}

\bibliographystyle{plain}
\title{
Classifying presentations of finite groups - the case of dicyclic
groups}

\author{
Peteris\ Daugulis\thanks{Daugavpils University, Daugavpils,
LV-5400, Latvia (peteris.daugulis@du.lv). } }

\pagestyle{myheadings} \markboth{ P.\ Daugulis}{Classifying
presentations of finite groups - the case of dicyclic groups}
\maketitle

\begin{abstract} The problem of classifying
equivalence classes of presentations up to isomorphism of Cayley
graphs is considered in this article in the case of dicyclic
groups. The number of equivalence classes of presentations is
uniformly bounded - it is a "finite presentation type" case. We
find all equivalence classes of presentations of dicyclic groups
having two generators. For the dicyclic group of order $4n$ apart
from the classical presentation with order multiset $\{\{2n,4\}\}$
for all $n$ there are presentations with order multiset
$\{\{4,4\}\}$. If $n$ is odd there is an additional presentation
having elements with order multiset $\{\{n,4\}\}$. These results
may be used in characterizing group structure and properties.
\end{abstract}

\begin{keywords}
group presentation, Cayley graph, dicyclic group, generalized
quaternion group
\end{keywords}

\begin{AMS}
20F05.
\end{AMS}

\section{Introduction and outline} Given a group $G$ with a
generating sequence $\mathcal{S}$ we define the edge-labeled
Cayley graph $\Gamma(G,\mathcal{S})$ in the standard way on the
vertex set $G$ with labeled directed edges corresponding to left
multiplication by elements of $\mathcal{S}$. Two presentations
$\langle \mathcal{S}_{1}|R_{1}\rangle$ and $\langle
\mathcal{S}_{2}|R_{2}\rangle$ (and corresponding generating
sequences $\mathcal{S}_{1}$,$\mathcal{S}_{2}$) are defined
equivalent iff $\Gamma(G,\mathcal{S}_{1})$ and
$\Gamma(G,\mathcal{S}_{2})$ are isomorphic as edge-labeled
directed graphs, up to edge relabelings.

In this article we solve the problem of finding all equivalence
classes of presentations with two generators for a series of
finite groups - dicyclic groups. The dicyclic groups are chosen as
one of the first cases of this problem for the author because in
this case the problem can be called "of finite presentation type"
using the analogy of linear representation theory - the number of
equivalence classes of presentations is uniformly bounded for all
orders. The main result of the article can be summarized in the
following theorem.

\begin{theorem}\label{12} Let $DC_{4n}=\langle a,x|a^{2n}=e, x^{2}=a^{n},
x^{-1}ax=a^{-1}\rangle, n\in \mathbb{N}, n\ge 2$ (the classical
presentation of this group).

\begin{enumerate}

\item If $2|n$ then there are two equivalence classes of minimal
presentations with two generators of $DC_{4n}$:  the classical
presentation and $\Pi_{4n,1}=\langle u,v|u^2=v^2,
u^4=u^2(u^3v)^{n}=e\rangle$.

\item If $2\not|n$ then there are four equivalence classes of
minimal presentations with two generators of $DC_{4n}$: the
classical presentation,
 $\Pi_{4n,1}=\langle u,v|u^2=v^2,
u^4=u^2(u^3v)^{n}=e\rangle$, $\Pi_{4n,0}=\langle u,v|u^2=v^2,
u^4=u^2(uv)^{n}=e\rangle$ and $\Pi_{4n,n}=\langle b,y|b^n=e,
y^4=e, y^{-1}by=b^{-1}\rangle$.

\end{enumerate}

\end{theorem}

In group theory one usually works with fixed (classical)
presentations. For important series of groups such as symmetric or
alternating groups, certain generators and presentations have been
accepted as standard ones.
Interesting problems of
finding generators with given low orders have been solved for
symmetric and alternating groups, see \cite{D}. Generators and
presentations of simple groups is an active research area, see
\cite{G}.

Definition of natural equivalence relations and classification of
equivalence classes of mathematical objects in any area is a
motivated, albeit often auxiliary, problem once these objects are
defined.  In algebra useful equivalence relations are defined
considering changes of generators of algebraic objects.

The problem of defining equivalence relations on sets of
presentations and describing all equivalence classes of
presentations does not seem to have been clearly formulated and
addressed in the literature.
In group theory classification of presentations may be related to
some general problems of group theory such as classification of
groups. This problem is trivial for extreme cases such as cyclic
or elementary abelian groups. Other cases may give additional
description of groups.

%
 A suitable graph-based technique is
introduced. All groups considered in this article are finite.

\section{Review}

\subsection{Sequences}\

Given sequences $\mathcal{A}_{i}=(a_{i1},...,a_{in_{i}})$, $i\in
\{1,...,m\}$ we define their concatenation
$\mathcal{A}=\mathcal{A}_{1}...\mathcal{A}_{m}$ in the standard
way. We assume that each $\mathcal{A}_{i}$ is a subsequence of
$\mathcal{A}$. Given two sequences
$\mathcal{A}=(a_{1},...,a_{m})$, $\mathcal{B}=(b_{1},...,b_{m})$,
the function $f:\mathcal{A}\rightarrow \mathcal{B}$ is a sequence
of assignments $f(a_{i})=b_{i}$.
Given a sequence $\mathcal{S}$ we
define $Set(\mathcal{S})$ to be the underlying set of
$\mathcal{S}$
. We denote union of multisets by
$\coprod$. Double curly brackets are used for multisets.

\subsection{Group presentations and Cayley graphs}

An edge-labeled graph is a quadruple $\Gamma=(V,E,k,w)$, where
 $k$ is the \sl set of edge labels\rm\  and $w:E\rightarrow
k$ is an \sl edge-label function,\rm\ $w(a,b)=w$ means that the
edge $a\rightarrow b$, also denoted as the ordered pair $(a,b)$,
is given the label $w$, in other notations
$a\stackrel{w}\rightarrow b$, $(a,b)_{w}$. We denote the
corresponding undirected edge-labeled graph by $\Gamma_{u}$.

Two edge-labeled graphs $\Gamma_{1}=(V_{1},E_{1},k_{1},w_{1})$ and
$\Gamma_{2}=(V_{2},E_{2},k_{2},w_{2})$ are isomorphic
($\Gamma_{1}\simeq \Gamma_{2}$) if there are two bijective
functions $f:V_{1}\rightarrow V_{2}$, $\sigma:k_{1}\rightarrow
k_{2}$ such that $a\stackrel{w}\rightarrow b$ iff
$f(a)\stackrel{\sigma(w)}\rightarrow f(b)$, for any pair $a,b$.
The graphs $\Gamma_{1}$ and $\Gamma_{2}$ are called
undirected-isomorphic ($\Gamma_{1}\simeq_{u} \Gamma_{2}$) if
$(\Gamma_{1})_{u}\simeq (\Gamma_{2})_{u}$, as undirected
edge-labeled graphs.


Let $G$ be a group. In this article we consider generating
sequences instead of traditional generating sets, relations are
still considered as sets. Let $\mathcal{S}$ be a sequence of
$G$-elements: $\mathcal{S}\in G^{l}$, $l=|\mathcal{S}|$
. We denote $\langle \mathcal{S}\rangle=\langle
Set(\mathcal{S})\rangle$ and $S=Set(\mathcal{S})$. Let
$E_{S}=\bigcup\limits_{g\in G,s\in S}(g,sg)_{s}$.
$w_{S}:E_{S}\rightarrow S$ is defined as follows: $w_{S}(g,sg)=s$.
The edge-labeled graph $\Gamma(G,\mathcal{S})=(G,E_{S},S,w_{S})$
is called \sl the Cayley graph\rm\ of $G$ with respect to the
sequence $\mathcal{S}$. If $G=\langle \mathcal{S}\rangle$ then
$\Gamma(G,\mathcal{S})$ is connected. For any two group elements
$g_{1}\in G, g_{2}\in G$ there is a unique edge-labeled graph
automorphism of $\Gamma(G,\mathcal{S})$ sending $g_{1}$ to
$g_{2}$. A group automorphism $\varphi:G\rightarrow G$ induces a
graph isomorphism (\sl Cayley isomorphism\rm)
$\Gamma(G,\mathcal{S})\rightarrow \Gamma(G,\varphi(\mathcal{S}))$.
See \cite{M}.



\subsection{Notations and review of dicyclic groups} In terms of \ref{12} denote $\mathbb{A}_{4n}=\bigcup_{i=0}^{2n-1}a^i$
and $\mathbb{X}_{4n}=\bigcup_{j=0}^{2n-1}a^jx$. We have that
$DC_{4n}=\mathbb{A}_{4n} \cup \mathbb{X}_{4n}$
. We note the following obvious
multiplication rules:
$(a^{k}x)(a^{m}x)=a^{k-m+n}$, $(a^{k}x)^{-1}=a^{k+n}x$.
$DC_{2^k}=Q_{2^k}$ is called \sl generalized quaternion group,\rm\
see \cite{R1}.




\section{Main results}\

\subsection{Graph-based equivalence relation of presentations}\


\begin{definition}\label{11}
$G$ - a group, $\mathcal{S}_{i}$ - $G$-generating sequences, $i\in
\{1,2\}$, $G=\langle \mathcal{S}_{i} \rangle$, $R_{i}$ - sets of
relations between elements of $\mathcal{S}_{i}$. The presentations
$\langle \mathcal{S}_{1}|R_{1}\rangle$ and $\langle
\mathcal{S}_{2}|R_{2}\rangle$ (and corresponding generating
sequences/sets) will be called \sl equivalent\rm\ (denoted
$\langle \mathcal{S}_{1}|R_{1}\rangle \simeq \langle
\mathcal{S}_{2}|R_{2}\rangle$) if $\Gamma(G,\mathcal{S}_{1})\simeq
\Gamma(G,\mathcal{S}_{2})$.
\end{definition}

Studying any group or a family of groups we may pose and solve the
problem of finding all equivalence types of minimal presentations.

\begin{example}
The group of minimal cardinality having two nonequivalent
presentations of the same number of generators is $\Sigma_{3}$. It
can be generated by any two elements of orders $2$ and $3$, or by
any two elements of order $2$. Computations show that symmetric
and alternating groups have more than one equivalence class of
minimal presentations: for $\Sigma_{4}$ there are $5$ classes of
minimal presentations with two generators and $9$ classes of
minimal presentations with three generators.
\end{example}

%

\begin{example} We assume it known that
for the dihedral group $D_{2n}$, $n\ge 3$, there are two
equivalence types of presentations with two elements - $\langle
a,x|a^n=x^2=e, xax=a^{-1} \rangle$ and $\langle
u,v|u^2=v^2=(uv)^{n}=e\rangle$, where $u=a^{k_{1}}x$ and
$v=a^{k_{2}}x$ with $GCD(k_{1}-k_{2},n)=1$. It can also be proved
by methods of this article. Thus the classification problem for
dihedral groups is also of "finite presentation type".

\end{example}

A sufficient condition for two presentations to be non-equivalent
is nonequality of multisets of generator orders. Given a sequence
$\mathcal{S}=(s_{1},..,s_{n})$, $s_{i}\in G$, define
$om(\mathcal{S})=\coprod_{i}^{n}Ord(s_{i})$ - \sl the order
multiset\rm\ of $\mathcal{S}$.

\begin{proposition}$G$ - a group, $\mathcal{S},\mathcal{T}$ - sequences of
$G$-elements. $\Gamma(G,S)\simeq \Gamma(G,T)$ implies
$om(\mathcal{S})=om(\mathcal{T})$.

\end{proposition}

\begin{proof} Vertices of oriented loops corresponding to relations $s^{k_{i}}_{i}=e_{G}$, $s_{i}\in
Set(\mathcal{S})$, $k_{i}=Ord(s_{i})$, are mapped by graph
isomorphisms to vertices of loops corresponding to relations
$t^{m_{j}}_{j}=e_{G}$, $t_{j}\in Set(\mathcal{T})$,
$m_{j}=Ord(t_{j})$, for some $j$. For each $i$ we must have
$k_{i}=m_{j}$, thus a Cayley graph isomorphism defines a function
$om(\mathcal{S})\rightarrow om(\mathcal{T})$ which permutes equal
elements. If $om(\mathcal{S})\ne om(\mathcal{T})$, then a
bijective function with such property is not possible.
\end{proof}

\begin{remark} Equality of generator order multisets is not a
sufficient condition for presentations to be equivalent. The
smallest group having at least two non-equivalent presentations
with two generators and the same order multiset is
$\mathbb{Z}_{3}\times \mathbb{Z}^{2}_{2}$, it has two
non-equivalent presentations each with order multiset
$\{\{6,6\}\}:=\{\{6^2\}\}$ (two elements of order $6$).

\end{remark}

Additionally we can define an equivalence relation using
isomorphism of undirected edge-labeled Cayley graphs.

\begin{definition}\label{10}
$G$, $\mathcal{S}_{i}$ as in Definition \ref{11}. The
presentations $\langle \mathcal{S}_{1}|R_{1}\rangle$ and $\langle
\mathcal{S}_{2}|R_{2}\rangle$ are called \sl
undirected-equivalent\rm\ (denoted $\langle
\mathcal{S}_{1}|R_{1}\rangle \simeq_{u} \langle
\mathcal{S}_{2}|R_{2}\rangle$) if
$\Gamma(G,\mathcal{S}_{1})\simeq_{u} \Gamma(G,\mathcal{S}_{2})$.
\end{definition}

\begin{example} $A_{4}$ can be generated by two $3$-cycles in two non-equivalent ways: $A_{4}=\langle (1,2,3),(2,4,3)\rangle=\langle
(1,2,3),(2,3,4)\rangle$, but these presentations are
undirected-equi\-va\-lent.

\end{example}

%

\subsection{Minimal generating sequences for dicyclic groups}

\subsubsection{Generating elements of order $2n$ and $4$}\label{4}
\

\begin{proposition}\label{6} Consider $DC_{4n}$ as defined in \ref{12}.
Any presentation with two generators with orders $2n$ and $4$ is
equivalent to the classical presentation.

\end{proposition}

\begin{proof}
$\mathbb{A}_{4n}$ is the only cyclic subgroup of order $2n$. If
$\mathbb{A}_{4n}=\langle b \rangle$, then $b=a^{t}$ with $t$
invertible mod $2n$, $DC_{4n}=\langle b,y\rangle$ for any
$y=a^mx\in \mathbb{X}_{4n}$. $b$ and $y$ satisfy relations
$b^{2n}=e$, $b^n=(a^n)^t=a^n=x^2=y^2$, $y^{-1}by=b^{-1}$, thus all
such presentation are equivalent.
\end{proof}

\subsubsection{Two generating elements of order $4$}\label{7}\

\begin{proposition}\label{1} Consider $DC_{4n}$ as defined in \ref{12}. Let $u=a^kx$,
$v=a^mx$. Then $\langle u,v\rangle=DC_{4n}$ iff $GCD(n,k-m)=1$.

\end{proposition}

\begin{proof} If $GCD(n,k-m)=1$, then there exist $\alpha,\beta\in
\mathbb{Z}$ such than $1=\alpha n+\beta (k-m)$. We have that
$u^2=a^{n}$ and $u^{-1}v=a^{k-m}$. It follows that
$a=u^{2\alpha}(u^{-1}v)^{\beta}\in \langle u,v\rangle$,
$x=a^{-k}u\in \langle u,v\rangle$, and thus $\langle
u,v\rangle=DC_{4n}$.

Let $GCD(n,k-m)=d>1$. We prove by induction that a proper subset
of $DC_{4n}$ is closed under generation by $\{u,v\}$ and contains
$\{u,v\}$, and thus $\langle u,v\rangle \ne DC_{4n}$.

We say that $Y\subseteq DC_{4n}$ is $d$-special (and contains
$d$-special elements) iff 1) $Y\cap \mathbb{A}_{4n}\subseteq
\langle a^d\rangle$ and 2) if $a^{k_{1}}x\in Y$ and $a^{k_{2}}x\in
Y$, then $k_{1}\equiv k_{2}(mod\ d)$. Note that $d|2n$, $d>1$,
implies that a $d$-special set is a proper subset of $DC_{4n}$.
Note that inverses of $d$-special elements are $d$-special.

Define $S_{0}=\{u,v\}$, note that $S_{0}$ is $d$-special.
Inductive hypothesis - suppose that after $k$ steps (adding
products) we generate a $d$-special subset $S_{k}\subseteq
DC_{4n}$. We prove that after $k+1$ steps we will get a
$d$-special set $S_{k+1}$. So se have to prove that a product of
two $d$-special elements is $d$-special: 1)
$a^{k_{1}d}a^{k_{2}d}=a^{(k_{1}+k_{2})d}\in \langle a^{d}\rangle$,
2) $a^{kd}(a^{t}x)=a^{kd+t}x$, we have that $t\equiv kd+t (mod\
d)$, 3) $(a^{t}x)a^{kd}=a^{t-kd+n}x$, we have that $t\equiv t-kd+n
(mod\ d)$, 4) let $t_{1}\equiv t_{2} (mod\ d)$, then
$(a^{t_{1}}x)(a^{t_{2}}x)=a^{t_{1}-t_{2}+n}$, we have that
$t_{1}-t_{2}+n\equiv 0 (mod\ d)$.

We have proved that from $S_{0}$ we can generate only $d$-special
subsets of $DC_{4n}$. Thus $d>1$ implies $\langle u,v\rangle\ne
DC_{4n}$.
\end{proof}

\begin{corollary} Since $Ord(a^kx)=4$, for any $k\in \mathbb{Z}$,
it follows from Proposition \ref{1} that for any $n\ge 2$ there
are presentations of $DC_{4n}$ having two elements of order $4$.

\end{corollary}


\begin{proposition}\label{5} Consider $DC_{4n}$ as defined in \ref{12}.

\begin{enumerate}

\item If $2|n$ then there is one equivalence type of
presentations, i.e. let $k_{i}\ne m_{i}$,
$\mathcal{S}_{i}=(a^{k_{i}}x,a^{m_{i}}x)$, $i\in \{1,2\}$
$GCD(n,k_{i}-m_{i})=1$, then $\Gamma(G,\mathcal{S}_{1})\simeq
\Gamma(G,\mathcal{S}_{2})$.

\item If $2\not|n$ then there is two equivalence types of
presentations, i.e. let $k_{i}\ne m_{i}$,
$\mathcal{S}_{i}=(a^{k_{i}}x,a^{m_{i}}x)$, $i\in \{1,2\}$,
$GCD(n,k_{i}-m_{i})=1$, then $\Gamma(G,\mathcal{S}_{1})\simeq
\Gamma(G,\mathcal{S}_{2})$ iff $k_{1}-m_{1}\equiv k_{2}-m_{2}(mod\
2)$.

\end{enumerate}

\end{proposition}

\begin{proof}
1. Let $2|n$, $G=DC_{4n}$. Choose $k\ne m$ such that
$GCD(n,k-m)=1$, $u=a^kx$, $v=a^mx$. Define the generating sequence
$\mathcal{S}=(u,v)$. We start constructing $\Gamma(G,\mathcal{S})$
from $e$ in the following steps.

\begin{enumerate}

\item[Step\ 0] Apply $u,u^2,u^3$ to $e$, get the set
$G_{0}=\{e,u,u^2,u^3\}$.

\item[Step\ 1] Apply first $v$, then $u,u^2,u^3$ to $e$, generate
$G_{1}=\{v,uv,u^2v,u^3v\}$. Find all $v$-edges between $G_{0}$ and
$G_{1}$. For all $k,m$ there are $4$ $v$-edges $$e\rightarrow
v\rightarrow u^2 \rightarrow u^2v\rightarrow e.$$

\item[Step\ 2] Apply first $v$, then $u,u^2,u^3$ to $u^3v$,
generate $$G_{2}=\{v(u^3v),uv(u^3v),u^2v(u^3v),(u^3v)^2\}.$$ Find
all $v$-edges between $G_{1}$ and $G_{2}$. For all $k,m$ there are
$4$ $v$-edges $$u^3v\rightarrow v(u^3v)\rightarrow uv \rightarrow
u^2v(u^3v)\rightarrow u^3v.$$

\item[...] ...

\item[Step\ n-1] Apply first $v$, then $u,u^2,u^3$ to
$(u^3v)^{n-2}$, generate
$$G_{n-1}=\{v(u^3v)^{n-2},uv(u^3v)^{n-2},u^2v(u^3v)^{n-2},(u^3v)^{n-1}\}.$$
Find all $v$-edges between $G_{n-2}$ and $G_{n-1}$. For all $k,m$
there are $v$-edges $(u^3v)^{n-2}\rightarrow
v(u^3v)^{n-2}\rightarrow uv(u^3v)^{n-3} \rightarrow
u^2v(u^3v)^{n-2}\rightarrow u^3v$.

\item[Step\ n] Find all $v$-edges between $G_{n-1}$ and $G_{0}$.
For all $k,m$ there are $v$-edges $(u^3v)^{n-1}\rightarrow
u^3\rightarrow uv(u^3v)^{n-2} \rightarrow u\rightarrow
(u^3v)^{n-1}$.

\end{enumerate}

We have that $$G_{0}=\{e,u,u^2,u^3\},
G_{j}=\{v(u^3v)^{j-1},uv(u^3v)^{j-1},u^2v(u^3v)^{j-1},u^3v(u^3v)^{j-1}\}$$
for $i\in \{1,...,n-1\}$, $G=\cup_{i=0}^{n-1}G_{i}$. There are no
other edges. We see that the Cayley graph construction is uniquely
determined for all $k,m$.

Let $\mathcal{S}_{i}=(a^{k_{i}}x,a^{m_{i}}x)$, $i\in \{1,2\}$.
Define $u_{i}=a^{k_{i}}x$, $v_{i}=a^{m_{i}}x$.  Define sequences
$\mathcal{G}_{0,i}=(e,u_{i},u^2_{i},u^3_{i})$,
$\mathcal{G}_{j,i}=(v_{i}(u^3_{i}v_{i})^{j-1},u_{i}v_{i}(u^3_{i}v_{i})^{j-1},u^2_{i}v_{i}(u^3_{i}v_{i})^{j-1},u^3_{i}v_{i}(u^3_{i}v_{i})^{j-1})$
for $j\in \{1,...,n-1\}$, $i\in \{1,2\}$.

By the uniqueness of Cayley graph construction it follows that the
bijective function $\varphi:G\rightarrow G$, defined by
$\varphi(e)=e$, $\varphi(\mathcal{G}_{j,1})=\mathcal{G}_{j,2}$,
for each $j \in \{0,...,n-1\}$, is an isomorphism between
$\Gamma(G,\mathcal{S}_{1})$ and $\Gamma(G,\mathcal{S}_{2})$ with
the edge-relabeling function $\sigma$ such that
$\sigma(u_{1})=u_{2}$, $\sigma(v_{1})=v_{2}$.

 2. Let $2\not|n$, $G=DC_{4n}$. Choose $k\ne m$
such that $GCD(n,k-m)=1$, $u=a^kx$, $v=a^mx$, $\mathcal{S}=(u,v)$.
Again we start constructing $\Gamma(G,\mathcal{S})$ from $e$.

First $n-1$ steps are the same as in proof of statement 1, the
construction is unique.

For the Step n there are 2 possibilities:

\begin{enumerate}

\item  if $k-m\equiv 1 (mod\ 2)$, then there are $4$ $v$-edges
$$(u^3v)^{n-1}\rightarrow u^3\rightarrow (uv)(u^3v)^{n-2}
\rightarrow u\rightarrow (u^3v)^{n-1},$$

\item if $k-m\equiv 0 (mod\ 2)$, then there are $4$ $v$-edges
$$(u^3v)^{n-1}\rightarrow u \rightarrow (uv)(u^3v)^{n-2}\rightarrow
u^3\rightarrow (u^3v)^{n-1}.$$

\end{enumerate}

There are no other edges.

If $k_{1}-m_{1}\equiv k_{2}-m_{2} (mod\ 2)$, then
$\Gamma(G,\mathcal{S}_{1})\simeq \Gamma(G,\mathcal{S}_{2})$ by the
same argument as in 1.

Let $k_{1}-m_{1}\not\equiv k_{2}-m_{2} (mod\ 2)$,
$S_{i}=\{a^{k_{i}}x,a^{m_{i}}x\}$, $i\in \{1,2\}$. Define
$u_{i}=a^{k_{i}}x$, $v_{i}=a^{m_{i}}x$. Suppose that
$k_{1}-m_{1}\equiv 1 (mod\ 2)$. We show by contradiction that in
this case $\Gamma(G,\mathcal{S}_{1})\not\simeq
\Gamma(G,\mathcal{S}_{2})$.

By construction
$G=\cup_{i=0}^{n-1}G_{1,i}=\cup_{i=0}^{n-1}G_{2,i}$. If
$\Gamma(G,\mathcal{S}_{1})\simeq \Gamma(G,\mathcal{S}_{2})$, then
there is an isomorphism $\psi:\Gamma(G,\mathcal{S}_{1})\rightarrow
\Gamma(G,\mathcal{S}_{2})$ fixing $e$. We show that it is
impossible. We have to consider two possible edge-relabeling
functions $\sigma$ and $\tau$:
$\sigma(u_{1},v_{1})=(u_{2},v_{2})$,
$\tau(u_{1},v_{1})=(v_{2},u_{2})$.

\paragraph{Case $\sigma$}

 We must have $\psi(\mathcal{G}_{1,i})=\psi(\mathcal{G}_{2,i})$,
 thus $\psi$ is completely determined. We check if $\psi$ maps
edges to edges mapping edge labels by $\sigma$. Considering
$v$-edges between $G_{i,0}$ and $G_{i,n-1}$ we get a
contradiction: it is sufficient to notice that there is a $v$-edge
$u\rightarrow (u^3v)^{n-1}$ in $\Gamma(G,\mathcal{S}_{1})$, but a
$v$-edge $u\rightarrow (uv)(u^3v)^{n-2}$ in
$\Gamma(G,\mathcal{S}_{2})$.

\paragraph{Case $\tau$} We generate $G$ in the same way as above
interchanging generators in the generating sequence. Again we get
two possibilities in the last step. By the same argument we have
$\Gamma(G,\mathcal{S}_{1})\not\simeq\Gamma(G,\mathcal{S}_{2})$.
\end{proof}

\begin{remark} Note that the underlying undirected graphs of both non-isomorphic directed
graphs in statement 2, \ref{5}, are isomorphic as edge-labeled
undirected graphs. Thus there is one equivalence class of
presentations with two generators of order $4$ in the sense of
Definition \ref{10}.

\end{remark}

\begin{proposition}
\begin{enumerate}

\item If $2|n$ then $DC_{4n}$ is isomorphic to $\Pi_{4n,1}=\langle
u,v|u^2=v^2, u^4=u^2(u^3v)^{n}=e\rangle$.

\item If $2\not|n$ then $DC_{4n}$ is isomorphic to
$\Pi_{4n,1}=\langle u,v|u^2=v^2, u^4=u^2(u^3v)^{n}=e\rangle$ and
$\Pi_{4n,0}=\langle u,v|u^2=v^2, u^4=u^2(uv)^{n}=e\rangle$.

\end{enumerate}

\end{proposition}

\begin{proof}
It is sufficient to exhibit group morphisms $\varphi_{i}:
DC_{4n}\rightarrow  \Pi_{4n,i}$ and $\psi_{i}: \Pi_{4n,i}
\rightarrow DC_{4n}$ such that

\begin{equation}\label{3} \psi_{i}\circ \varphi_{i}=id_{DC_{4n}},\ \varphi_{i}\circ\psi_{i}=id_{\Pi_{4n,i}}.
\end{equation}

1. Define $\varphi_{1}$ and $\psi_{1}$ on the generators:
$$\left\{%
\begin{array}{ll}
    \varphi_{1}(a)=u^3v \\
    \varphi_{1}(x)=v \\
\end{array}%
\right.  and\
\left\{%
\begin{array}{ll}
    \psi_{1}(u)=ax \\
    \psi_{1}(v)=x \\
\end{array}%
\right.
 .$$

 2. Additionally we define $\varphi_{0}$ and $\psi_{0}$ on the generators:
$$\left\{%
\begin{array}{ll}
    \varphi_{0}(a)=vu \\
    \varphi_{0}(x)=v \\
\end{array}%
\right.  and\
\left\{%
\begin{array}{ll}
    \psi_{0}(u)=a^{n-1}x \\
    \psi_{0}(v)=x \\
\end{array}%
\right.
 .$$

 We check that $\varphi_{i}$ and $\psi_{i}$ can be extended to group morphisms and \ref{3}
 hold. For $\varphi_{0}$ and $\psi_{0}$ we take into account that
 $n$ is odd.
\end{proof}

%

\begin{example} Any presentation of $DC_{12}$ having two elements
of order $4$ is equivalent either to presentation

\begin{enumerate}

\item[1)] $\Pi_{12,1}=\langle u,v|u^4=e, u^2=v^2,
u^2(u^3v)^3\rangle$, or $\langle ax,x \rangle$, or

\item[2)] $\Pi_{12,0}=\langle u,v|u^4=e, u^2=v^2,
u^2(uv)^3\rangle\rangle$, $\langle a^2x,x\rangle$.

\end{enumerate}
$\Gamma(DC_{12},(ax,x))$ is shown in Fig.1. continuous arrows mean
left multiplication by $u=ax$ and dotted arrows mean left
multiplication by $v=x$.

\begin{center}
$ \xymatrix@R=0.04pc@C=0.04pc{
&\ar@{.>}[ddl]&&&&&&&&&&&\\
&&&&&&e\ar[drrrrrr]\ar@{.>}[dddddrrrrrr]&&&&&\ar@{.>}[dr]&\\
u^3\ar[urrrrrr]\ar@{.>}[ur]&&&&&&&&&&&&u\ar[dllllll]\ar@{.>}[uul]\\
&&&&&&u^2\ar[ullllll]\ar@{.>}[dddllllll]&&&&&&\\
&&&&&&&&&&&&\\
&&&&&&u^3v\ar[drrrrrr]\ar@{.>}[dddddrrrrrr]&&&&&&\\
u^2v\ar[urrrrrr]\ar@{.>}[uuuuurrrrrr]&&&&&&&&&&&&v\ar[dllllll]\ar@{.>}[uuullllll]\\
&&&&&&uv\ar[ullllll]\ar@{.>}[dddllllll]&&&&&&\\
&&&&&&&&&&&&\\
&&&&&&(u^3v)^2\ar[drrrrrr]\ar@{.>}[dl]&&&&&&\\
u^2vu^3v\ar[urrrrrr]\ar@{.>}[uuuuurrrrrr]&&&&&&&\ar@{.>}[ul]&&&&&vu^3v\ar[dllllll]\ar@{.>}[uuullllll]\\
&&&&&&uvu^3v\ar[ullllll]\ar@{.>}[dr]&&&&&&\\
&&&&&\ar@{.>}[ur]&&&&&&&\\
} $

Fig.1. - the graph $\Gamma(DC_{12},(ax,x))$.

\end{center}

\end{example}

%
%
%
%

%

\subsubsection{Generating elements of order $n$ and $4$}\label{8}\

\begin{proposition} Consider $DC_{4n}$ as defined in \ref{12}.
If $DC_{4n}=\langle b,y\rangle$, $b\in \mathbb{A}_{4n}$, $y\in
\mathbb{X}_{4n}$, then there are two possibilities:

\begin{enumerate}

\item $Ord(b)=2n$ or

\item $2\not |n$ and $Ord(b)=n$.

\end{enumerate}

\end{proposition}

\begin{proof}
If $\langle v,z\rangle=DC_{4n}$ then $v\in \mathbb{X}_{4n}$ or
$z\in \mathbb{X}_{4n}$. Let $z=a^{m}x\in \mathbb{X}_{4n}$. The
case $v\in \mathbb{X}_{4n}$ has already been discussed. Let
$v=a^k$. By an inductive argument similar to that in the proof of
\ref{1}, we can show that $\langle a^{k},a^{m}x\rangle=DC_{4n}$
iff $GCD(n,k)=1$.

For $GCD(n,k)=1$ there are two possibilities: $GCD(2n,k)=1$ or
$GCD(2n,k)=2$ and $2\not|n$ . The case $GCD(2n,k)=1$ has been
discussed in subsection \ref{4}.

If $GCD(2n,k)=2$, then $k=2^tq'$, $GCD(n,q')=1$. If $Ord(a^k)=r$,
then $2^tq'r\equiv 0 (mod\ 2n)$.  It follows that $n|r$, thus
$Ord(a^k)=n$ in this case.
\end{proof}

\begin{proposition}\label{9} Consider $DC_{4n}$ as defined in \ref{12}, $2\not |n$, $n\ge 3$. Then

\begin{enumerate}

\item $DC_{4n}=\langle a^2, x\rangle$,

\item  $DC_{4n}\simeq\Pi_{4n,n}$, where $\Pi_{4n,n}=\langle
b,y|b^n=e, y^4=e, y^{-1}by=b^{-1}\rangle$,

\item If $\mathcal{S}=(b,y)$, $b\in \mathbb{A}_{4n},y\in
\mathbb{X}_{4n}$, $Ord(b)=n$, $Ord(y)=4$, then
$\Gamma(DC_{4n},\mathcal{S})\simeq \Gamma(DC_{4n},(a^2,x))$.


\end{enumerate}

\end{proposition}

\begin{proof}
1. $2\not|n$ implies $a^{n+1}\in \langle a^2\rangle$. $a^n=x^2$
implies $a\in \langle a^2,x \rangle$ and $\langle
a^2,x\rangle=DC_{4n}$.

2. We exhibit group morphisms $\varphi:DC_{4n}\rightarrow
\Pi_{4n,n}$, $\psi:\Pi_{4n,n}\rightarrow DC_{4n}$, satisfying
identities similar to \ref{3}.

Define $\varphi$ and $\psi$ on the generators:
$$\left\{%
\begin{array}{ll}
    \varphi(a)=b^{q}y^2, where\ q=\frac{n+1}{2} \\
    \varphi(x)=y \\
\end{array}%
\right.  and\
\left\{%
\begin{array}{ll}
    \psi(b)=a^{2} \\
    \psi(y)=x \\
\end{array}%
\right. .$$

It is directly checked that $\varphi$ and $\psi$ can be extended
to group morphisms from generators and identities
$\psi\circ\varphi=id_{DC_{4n}}$,
$\varphi\circ\psi=id_{\Pi_{4n,n}}$ are satisfied.

3. It follows that $y^{-1}by=b^{-1}$. Existence of a graph
isomorphism follows from the uniqueness of Cayley graph
construction as in the proof of Proposition \ref{5}.
\end{proof}

%
%

%
%

\begin{example} In Fig.2. continuous arrows mean left multiplication by $a^2$ and dotted arrows mean
left multiplication by $x$.
\begin{center}
$ \xymatrix@R=0.04pc@C=0.04pc{
e\ar[drrrrr]\ar@{.>}[ddd]&&&&&&&&&&a^4\ar[llllllllll]\ar@{.>}[ddd]\\
&&&&&a^2\ar[urrrrr]\ar@{.>}[ddd]&&&&&\\
&&&&&&&&&&\\
x\ar[rrrrrrrrrr]\ar@{.>}[ddd]&&&&&&&&&&a^2x\ar[dlllll]\ar@{.>}[ddd]\\
&&&&&a^4x\ar[ulllll]\ar@{.>}[ddd]&&&&&\\
&&&&&&&&&&\\
a^3\ar[drrrrr]\ar@{.>}[ddd]&&&&&&&&&&a\ar[llllllllll]\ar@{.>}[ddd]\\
&&&&&a^5\ar[urrrrr]\ar@{.>}[ddd]&&&&&\\
&&&&&&&&&&\\
a^3x\ar[rrrrrrrrrr]\ar@/^6ex/@{.>}[uuuuuuuuu]&&&&&&&&&&a^5x\ar[dlllll]\ar@/^6ex/@{.>}[uuuuuuuuu]\\
&&&&&ax\ar[ulllll]\ar@/^6ex/@{.>}[uuuuuuuuu]&&&&&\\
 } $

Fig.2. - the graph $\Gamma(DC_{12},(a^2,x))$.

\end{center}
\end{example}

\begin{remark} By the Burnside Basis Theorem all minimal generating sequences of $DC_{2^l}$ must have length $2$.
If $n$ is not a prime power, then $DC_{4n}$ has minimal generating
sequences containing more than two elements. For $DC_{24}$ there
are at least $6$ non-equivalent minimal generating sets containg
three elements. For example, $DC_{24}$ has the following minimal
generating sets: $\{a^{2},a^{3},x\}$ with orders $6,4,4$,
$\{a^{3},a^{4},x\}$ with orders $4,3,4$, $\{a^{3},x,a^{2}x\}$ with
orders $4,4,4$.
\end{remark}

\section*{Acknowledgement} Computations were performed using the
computational algebra system MAGMA, see Bosma et al. \cite{B3}.


\end{document}